\documentclass[1p]{elsarticle}

\usepackage{amsthm}
\usepackage{amssymb}
\usepackage{amsmath}
\usepackage{ytableau}
\usepackage{tikz-cd}
 \newtheorem{thm}{Theorem}[section]
 \newtheorem{cor}[thm]{Corollary}
 \newtheorem{lem}[thm]{Lemma}
 
 \theoremstyle{definition}
 \newtheorem{defn}[thm]{Definition}
 \newtheorem{rem}[thm]{Remark}
 \newtheorem*{rem*}{Remark}
 \newtheorem{rems}[thm]{Remarks}

  \newtheorem{exs}[thm]{Examples}
 \numberwithin{equation}{section}
\DeclareMathOperator{\Ima}{Im}
\DeclareMathOperator{\Hom}{Hom}
\DeclareMathOperator{\coker}{Coker}
\begin{document}

\title{Presentations of Schur and Specht modules in characteristic zero}


\author[1]{Mihalis Maliakas\corref{cor1}%
}
\ead{mmaliak@math.uoa.gr}

\author[2]{Maria Metzaki}
\ead{mmetzaki@math.uoa.gr}

\author[3]{Dimitra-Dionysia Stergiopoulou}
\ead{dstergiop@math.uoa.gr}

\cortext[cor1]{Corresponding author}

\affiliation[1]{organization={	Department of Mathematics,University of Athens},
	city={Athens},
	country={Greece}}
\affiliation[2]{organization={	Department of Mathematics,University of Athens},
	city={Athens},
	country={Greece}}
\affiliation[3]{organization={	School of Applied Mathematical and Physical Sciences, \\ National Technical University of Athens},
	city={Athens},
	country={Greece}}

\date{}


\begin{abstract}
New presentations of Specht modules of symmetric groups over fields of characteristic zero have been obtained by Brauner, Friedmann, Hanlon, Stanley and Wachs. These involve generators that are column tabloids and relations that are Garnir relations with maximal number of exchanges between consecutive columns or symmetrization of Garnir relations with minimal number of exchanges between consecutive columns. In this paper, we examine Garnir relations and their symmetrization with any number of exchanges. In both cases, we provide sufficient arithmetic conditions so that the corresponding quotient is a Specht module. In particular, in the first case this yields new presentations of Specht modules if the parts of the conjugate partition that correspond to maximal number of exchanges greater than 1 are distinct. These results generalize the presentations mentioned above and offer an answer to a question of Friedmann, Hanlon and Wachs. Our approach is via representations of the general linear group.
\end{abstract}


\maketitle
\section{Introduction} Let $K$ be a field of characteristic zero and let $\mathfrak{S}_r$ be the symmetric group on $r$ symbols. The irreducible representations of $\mathfrak{S}_r$ over $K$ are parametrized by partitions $\lambda$ of $r$. The irreducible module corresponding to $\lambda$ is denoted by $S^{\lambda}$ and is called the Specht module. Various constructions of $S^{\lambda}$ are known. We are interested in presentations with generators that are column tabloids.

For a partition $\lambda$ of $r$, a \textit{Young tableau} of shape $\lambda$ is a filling of the Young diagram of $\lambda$ with distinct entries from the set $\{1,2, \dots, r\}$.  Let $\mathcal{T}_\lambda$ be the set of Young tableaux of shape $\lambda$. Following \cite[7.4]{F}, let $M^\lambda $ be the $K$-vector space generated by $\mathcal{T}_\lambda$ subject to the relations $T+S=0$, if $S$ is obtained from $T \in \mathcal{T}_\lambda$ by a transposition of two entries in the same column of $T$. (In \cite{F} the notation $\tilde{M}^\lambda$ is used.) The symmetric group $\mathfrak{S}_r$ acts on $\mathcal{T}_\lambda$ by replacing each entry $i$ by $\sigma(i)$, where $\sigma \in \mathfrak{S}_r$. It follows that $M^\lambda$ is an $\mathfrak{S}_r$-module. For $T \in \mathcal{T}_\lambda$, let $\overline{T} \in M^\lambda$ be the corresponding coset. The elements $\overline{T} \in M^\lambda$ are called \textit{column tabloids}. It is clear these generate $M^\lambda$ as a vector space.

Given a partition $\lambda = (\lambda_1, \dots, \lambda_s)$ of $r$, let us fix $T \in \mathcal{T}_\lambda$, a column $c$ of $T$, where $c \le \lambda_1-1$, and a nonnegative integer $k$ such that $k \le \lambda^{\prime}_{c+1}$, where $\lambda^\prime =(\lambda_1^\prime, \dots, \lambda_{\lambda_1}^{\prime})$ is the conjugate partition of $\lambda$. The corresponding \textit{Garnir relation} is \begin{equation}\label{GarRel}g_{c,k}^T=\overline{T} -\sum \overline{S},\end{equation} where the sum is over all $S \in \mathcal{T}_\lambda$ obtained from $T$ by exchanging the top $k$ entries of column $c+1$ with any $k$ entries from column $c$ of $T$ while maintaining the vertical order of each of the exchanged sets. Let $G^\lambda$ be the subspace of $M^\lambda$ spanned by all 
$g_{c,k}^T$ as $T, c, k$ vary. (In \cite{F},  the notation $Q^\lambda$ is used.) Then $G^\lambda$ is an $\mathfrak{S}_r$-submodule of $M^\lambda$ and we have the important fact that $M^\lambda / G^\lambda \simeq S^\lambda$, see \cite[p. 99]{F}. Since the characteristic of $K$ is equal to zero, $G^\lambda$ is spanned by the $g_{c,1}^T$ as $T, c$ vary \cite[p. 102]{F}.

\subsection{Motivation} In their study of an $n$-ary generalization of the free Lie algebra termed free LAnKe, Friedmann, Hanlon, Stanley and Wachs obtained in \cite[Corollary 3.2]{FHSW} a new presentation of $S^\lambda$ when $\lambda$ is a staircase partition, i.e. a partition with conjugate of the form $(d,d-1,d-2,\dots,d-e)$. The relations are Garnir relations with maximal number of exchanges. This was recently generalized to more partitions, that include those whose conjugate has distinct parts, by Friedmann, Hanlon and Wachs in \cite{FHW}. Let $G^{\lambda, max}$ be the subspace of $M^\lambda$ spanned by all $g_{c,k}^T$, where $T \in \mathcal{T}_\lambda$, $c$ is a column of $T$ and $k=\lambda^{\prime}_{c+1}$. So here it is required that $k$ takes on the largest possible value for each column $c+1$ of $T$. We have $G^{\lambda, max} \subseteq G^\lambda$. The following result was shown in \cite{FHW}.
\begin{thm}[{\cite[Theorem 1.1]{FHW}}]\label{FHW} If $\lambda$ is a partition of $r$ satisfying $\lambda'_c > \lambda'_{c+1}$ for those $c$ such that  $\lambda'_{c+1} >1$, then as $\mathfrak{S}_r$-modules we have $M^\lambda / G^{\lambda, max} \simeq S^\lambda$.
\end{thm}

Brauner and Friedmann obtained in \cite{BF} a new presentation  of  $S^{\lambda}$ valid for all partitions $\lambda$. Here the relations are not Garnir relations, but sums of Garnir relations corresponding to $k=1$ and are obtained by a symmetrization process on the second column for each consecutive pair of columns. To be precise, fix  $T \in \mathcal{T}_\lambda$ and a column $c$ of $T$, where $c \le \lambda_1-1$. Let \begin{equation}h_{c}^{T}=\lambda_{c+1}^\prime \overline{T}-\sum\overline{S}, \end{equation}
where the sum ranges over all tableaux $S$ obtained from $T$ by exchanging one entry in column $c+1$ of $T$ with one entry in  column $c$. Let $SGR^{\lambda, min}$ be the subspace of $M^\lambda$ spanned by all $h_{c}^{T}$, where $T \in \mathcal{T}_\lambda$ and $c$ is a column of $T$. (In \cite{BF}, $SGR^{\lambda, min}$ is denoted by $\tilde{H}_\lambda $.) The following result was shown in \cite{BF}.
\begin{thm}[{\cite[Theorem 3.5]{BF}}]\label{BF} If $\lambda$ is a partition of $r$, then as $\mathfrak{S}_r$-modules we have $M^\lambda / SGR^{\lambda, min} \simeq S^\lambda.$
\end{thm}

\subsection{Results}\label{results} In this paper, we consider (a) relations that are symmetrizations of Garnir relations for any fixed $k$ and (b) Garnir relations for any fixed $k$. We allow $k$ to vary for different pairs of consecutive columns.

In case (a), fix positive integers $k_1,\dots, k_{\lambda_1-1}$ satisfying $k_c \le \lambda^{\prime}_{c+1}$ for all $c=1,\dots,{\lambda_1}-1$. If $T \in \mathcal{T}_\lambda$ and $c$ is a column of $T$, where $c \le {\lambda_1}-1$, we define \begin{equation}\label{newrel}
	r^{T}_{c,k_{c}}=\tbinom{\lambda_{c+1}^\prime}{k_c} \overline{T}-\sum\overline{S},
\end{equation}
where the sum is over all $S \in \mathcal{T}_\lambda$ obtained from $T$ by exchanging any $k_c$ entries of column $c+1$ with any $k_c$ entries from column $c$ of $T$ while maintaining the vertical order of each of the exchanged sets. Let $SGR^\lambda(k_1,\dots,k_{{\lambda_1}-1})$ be the subspace of $M^\lambda$ spanned by all $r^{T}_{c,k_{c}}$, where $T \in \mathcal{T}_\lambda$ and $c$ is a column of $T$ (symmetrized Garnir relations). In particular, we have two extreme cases; if $k_1=\dots=k_{{\lambda_1}-1}=1$, then $SGR^\lambda(1,\dots,1)=SGR^{\lambda, min}$, and  if  $k_1=\lambda^{\prime}_2,\dots,k_{{\lambda_1}-1}=\lambda^{\prime}_{\lambda_1}$, then $SGR^\lambda(\lambda^{\prime}_2,\dots,\lambda^{\prime}_{\lambda_1})=G^{\lambda,max}$. The subspace $SGR^\lambda(k_1,\dots,k_{\lambda_1-1})$ is an $\mathfrak{S}_r$-submodule of $M^\lambda$. In Theorem \ref{mainSpe}, we obtain sufficient arithmetic conditions on $\lambda$ and $k_1,\dots,k_{{\lambda_1}-1}$ so that $M^{\lambda} / SGR^\lambda(k_1,\dots,k_{{\lambda_1}-1}) \simeq S^\lambda$. 
\begin{thm}\label{1.3}
	Let $\lambda$ be a partition of $r$ and let $\mu$ be the conjugate partition of $\lambda$. Then as $\mathfrak{S}_r$-modules we have $M^{\lambda} / SGR^\lambda(k_1,\dots,k_{\lambda_1-1}) \simeq S^\lambda$ if \[\sum_{t=1}^{j}(-1)^{t-1}\tbinom{\mu_{c+1}-t}{\mu_{c+1}-k_c}\tbinom{j}{t}\tbinom{\mu_{c}-\mu_{c+1}+j+t}{t} \neq0\] for all $c=1,\dots,\lambda_1-1$ and  $j=1,\dots,\mu_{c+1}$.
\end{thm}
This generalizes Theorem \ref{FHW} and Theorem \ref{BF} (see Examples \ref{exs}(1),(2)).

In case (b), again fix positive integers $k_1,\dots, k_{\lambda_{1}-1}$ satisfying $k_c \le \lambda^{\prime}_{c+1}$ for all $c=1,\dots,{\lambda_1}-1$. Let $GR^\lambda(k_1,\dots,k_{{\lambda_1}-1})$ be the subspace of $M^\lambda$ spanned by all Garnir relations  defined in (\ref{GarRel}) of the form $g^{T}_{c,k_{c}}$, where $T \in \mathcal{T}_\lambda$ and $c$ is a column of $T$. Note that here, in contrast to the definition of $G^\lambda$, for any pair of consecutive columns $c$ and $c+1$, only exchanges of $k_c$ elements are allowed, where $k_c$ is fixed. So $SGR^\lambda(k_1,\dots,k_{\lambda_1-1}) \subseteq GR^\lambda(k_1,\dots,k_{{\lambda_1}-1}) \subseteq G^\lambda$.  We obtain the following generalization of Theorem \ref{FHW} (see Theorem \ref{thmmainSpGR}(2)).

\begin{thm}\label{1.4}
	If $\lambda$ is a partition of $r$ satisfying $\lambda'_c > \lambda'_{c+1}$ for those $c$ such that $\lambda'_{c+1}=k_c >1$, then  as $\mathfrak{S}_r$-modules $M^{\lambda} / GR^\lambda(k_1,\dots,k_{\lambda_1-1}) \simeq S^\lambda.$
\end{thm}

Theorem \ref{1.4} provides an answer to the question raised at the end of \cite{FHW} that asks for which partitions $\lambda$ we have $GR^\lambda(k_1,\dots,k_{\lambda_1-1}) =G^\lambda$.

Even though we approach both problems using representations of the general linear group $G=GL(n,K)$, our main motivation are the results and methods of \cite[Section 2]{FHSW}, \cite{BF} and \cite{FHW}, and the above question in \cite{FHW}. In these papers, certain $\mathfrak{S}_r$-equivariant linear operators on $M^\lambda$ were considered, where $\lambda$ consists of two columns, and their actions on irreducible summands of $M^\lambda$ were studied. In case (a), we obtain a result for all linear operators $\Phi \in \Hom_G(\Lambda^a \otimes \Lambda^b, \Lambda^a \otimes \Lambda^b )$, where $\Lambda^i$ is the $i$-th exterior power of the natural $G$-module of column vectors. This is accomplished by examining the action of elements of a particular basis of $\Hom_G(\Lambda^a \otimes \Lambda^b, \Lambda^a \otimes \Lambda^b )$ on each irreducible summand of $\Lambda^a \otimes \Lambda^b$. In Theorem \ref{main}, necessary and sufficient arithmetic conditions so that $\coker(\Phi)$ is irreducible are obtained. We regard this as the first main result of the paper. We then apply Theorem \ref{main} for a particular choice of operators considered in Definition \ref{maps2}. Descending to the symmetric group via the Schur functor, we obtain the relations $SGR^\lambda(k_1,\dots,k_{{\lambda_1}-1})$ and Theorem \ref{1.3} above, which we regard as our second main result. In case (b), we consider a similar analysis but only for particular maps $\gamma_k \in \Hom_G(\Lambda^a \otimes \Lambda^k \otimes \Lambda^{b-k}, \Lambda^a \otimes \Lambda^b )$ that correspond to Garnir relations (see Section 5). Our main results here are Theorem \ref{maincorGarnir3} and Theorem \ref{thmmainSpGR}.

Section 2 is devoted to preliminaries that are used in the sequel. In Section 3 we consider presentations of Schur modules for $G$. The main result is applied in  Section 4 to presentations  that are associated to symmetrizations of Garnir relations. In Section 5 we consider presentations associated to usual Garnir relations.  In Section 6 we derive consequences of our previous results for Specht modules.

\begin{rem*}
This version of the paper differs from the published version as follows. A typo in the statement of Corollary 4.4 and Corollary 6.1 has been corrected, instead of $j=1,\dots, b$, it should be $j=1, \dots, k$. Also, in the statement of Theorem 6.2 it should be $j=1,\dots, \mu_{c+1}$. Examples 6.3(1) and 6.3(3) have been corrected accordingly. We are indebted  to Tamar Friedmann for pointing out these corrections.
\end{rem*}

\section{Preliminaries}

As in the Introduction, let $G=GL(n,K)$, where $K$ is a field $K$ of characteristic zero, let $V$ be the natural $G$-module of column vectors, and let $\Lambda =\sum_{i\ge0}\Lambda^{i}$ be the exterior algebra of $V$. If $\alpha=(\alpha_1,...,\alpha_\ell)$ is a sequence of nonnegative integers, we denote the tensor product  $\Lambda^{a_1}\otimes \cdots \otimes \Lambda^{a_\ell}$ by $\Lambda^{\alpha}$.

The Schur module of $G$ corresponding to a partition $\mu $ will be denoted $L_{\mu}$, see \cite[2.1]{W}. This is defined as a quotient of $\Lambda^{\mu}$. Let us denote the corresponding projection by $\pi_{\mu} : \Lambda^{\mu} \to L_{\mu}$. With respect to the action of the diagonal matrices in $G$, $L_{\mu}$ has highest weight $\mu^{\prime}$, the conjugate partition of $\mu$. The formal character $ch(L_{\mu})$ of $L_{\mu}$ is the classical Schur polynomial $s_{\mu^\prime}$, \cite[2.2]{F}. For example, if $\mu = (a)$ consists of one part, then $ch(L_{(a)})$ is $s_{(1^a)}$, the elementary symmetric polynomial of degree $a$. We note that our $L_\mu$ is isomorphic to $E^\lambda$ of \cite{F}, where $\lambda = \mu^\prime$.

Since the characteristic of $K$ is zero, every finite dimensional representation of $G$ is a direct sum of irreducibles and the  $L_{\mu}$, where $\mu$ is a partition with first part at most $n$, form a complete set of irreducible polynomial representations of $G$.

Next we recall that $\Lambda$ has the structure of a Hopf algebra.  We will make use of the comultiplication $\Delta : \Lambda \to \Lambda \otimes \Lambda$ and the multiplication $m :  \Lambda \otimes \Lambda \to \Lambda$ of $\Lambda$. Multiplication in $\Lambda$ will be denoted by juxtaposition. To be explicit, the indicated component $ \Delta_{a,b} : \Lambda^{a+b} \to \Lambda^a \otimes \Lambda^b $ of $\Delta$ sends $v_1v_2 \cdots v_{a+b}$ to $\sum_{\sigma}sgn(\sigma)v_{\sigma(1)}\cdots v_{\sigma(a)} \otimes v_{\sigma(a+1)}\cdots v_{\sigma(a+b)}$, where $v_i \in V$ and the sum is over all permutations $\sigma$ of $\{1,\dots,a+b\}$ such that $\sigma(1) < \cdots < \sigma(a)$ and $\sigma(a+1) < \cdots < \sigma(a+b)$. In the computations that we make in the next sections, we will use several times the following fact. First let us establish some notation, which is a slight variation of the Sweedler notation, see \cite[p. 8]{GR}. For $x \in \Lambda^a$ we denote the image $\Delta_{a-t,t}(x)$ of $x$ under the map $\Delta_{a-t,t}: \Lambda^a \to \Lambda^{a-t} \otimes \Lambda^t$  by $\sum_{u}x_u(a-t)\otimes x'_u(t)$. Now let $x \in \Lambda^{a_1}, y \in \Lambda^{a_2}$, where $a=a_1+a_2$. Then $\Delta_{a-t,t}(xy)$ is equal to \begin{equation}\label{comul}\sum_{u,v}\sum_{t_1+t_2=t}(-1)^{(a_2-t_2)t_1}x_u(a_1-t_1)y_v(a_2-t_2)\otimes x'_u(t_1)y'_v(t_2),
\end{equation}
where the right sum is over all nonnegative integers $t_1,t_2$ such that $t_1+t_2=t$ and $t_i \le a_i$ for $i=1,2$. This may be proven by a straightforward induction on $a_2$.

We will use the following maps in Section 3.
\begin{defn}\label{maps} Let $a,b,t$ be nonnegative integers. Define the maps \begin{align*} &\delta_t : \Lambda^a \otimes \Lambda^b \xrightarrow{1 \otimes \Delta_{t,b-t}} \Lambda^a \otimes \Lambda ^t \otimes \Lambda ^{b-t}\xrightarrow{m \otimes 1}\Lambda^{a+t}  \otimes \Lambda ^{b-t}, \\&
		\theta_t:	\Lambda^a \otimes \Lambda^b \xrightarrow{\Delta_{a-t,t} \otimes 1 } \Lambda^{a-t} \otimes \Lambda ^t \otimes \Lambda ^{b}\xrightarrow{1 \otimes m}\Lambda^{a-t}  \otimes \Lambda ^{b+t}, \\& \Phi_t : \Lambda^a \otimes \Lambda^b \xrightarrow{\delta_{t}} \Lambda^{a+t} \otimes \Lambda ^{b-t}\xrightarrow{\theta_t}\Lambda^{a}  \otimes \Lambda ^{b}.	\end{align*}
\end{defn}These are maps of $G$-modules. In particular, for $t=0$ we have that $\delta_0 = \theta_0 = \Phi_0 = 1_{\Lambda^{a}\otimes \Lambda^{b}}$ is the identity map. It is well known, see for example \cite[Chapter 8, Exercise 13]{F}, that since the characteristic of $K$ is 0, we have the following presentation of $L_\mu$. The length $\ell =\ell(\mu)$ of a partition $\mu$ is the number of its nonzero parts. \begin{thm}\label{cok1}
	Let $\mu$ be a partition of length $\ell$. The cokernel of the map $ \theta_{\mu}  :\bigoplus_{c=1}^{\ell-1} \Lambda^{\mu_1} \otimes \cdots \otimes \Lambda^{\mu_c+1}\otimes \Lambda^{\mu_{c+1}-1} \otimes \cdots \otimes \Lambda^{\mu_{\ell}} \to \Lambda^{\mu}$ is isomorphic to $L_{\mu}$, where $\theta_{\mu}=\sum_{c=1}^{\ell-1} 1\otimes \cdots \otimes 1 \otimes \theta_1 \otimes 1 \otimes \cdots \otimes 1$.
\end{thm}

It is straightforward to verify that in the diagrams
\begin{center}
	\begin{tikzcd}[row sep=small,column sep=small]
		& \Lambda^{a+t}\otimes \Lambda^{b-t} \arrow[dr,"\delta_s"] & \\
		\Lambda^{a}\otimes \Lambda^{b} \arrow[ur, "\delta_t"]\arrow[rr,"{\delta_{t+s}}"] & & \Lambda^{a+t+s}\otimes \Lambda^{b-t-s}
	\end{tikzcd}
	
	\begin{tikzcd}[row sep=small,column sep=small]
		& \Lambda^{a-t}\otimes \Lambda^{b+t} \arrow[dr,"\theta_s"] & \\
		\Lambda^{a}\otimes \Lambda^{b} \arrow[ur, "\theta_t"]\arrow[rr,"{\theta_{t+s}}"] & & \Lambda^{a-t-s}\otimes \Lambda^{b+t+s}
	\end{tikzcd}
\end{center}
we have \begin{equation}\label{comm}
	\delta_{s}\circ\delta_{t}=\tbinom{s+t}{t}\delta_{s+t} \;  \text{and} \; 	\theta_{s}\circ\theta_{t}=\tbinom{s+t}{t}\theta_{s+t}.	
\end{equation}	

We recall here an  important property of $L_{\mu}$. We will only need the special case corresponding to two-rowed partitions. Let us fix the order $e_ 1 < e_2 < ... < e_n$ on the set $\{e_1, e_2, ..., e_n\}$ of the canonical basis elements of $V$ and let us denote $e_i$ by its subscript $i$. For a partition $\mu=(\mu_1,\mu_2)$, a \textit{tableau} of shape $\mu$ is a filling of the diagram of $\mu$ with entries from $\{1,...,n\}$. We call such a tableau \textit{cosemistandard} if the entries are strictly increasing in the rows from left to right and weakly increasing in the columns from top to bottom. This is perhaps not standard terminology, but we will keep in line with the fact that $L_{\mu}$  has highest weight $\mu^\prime$ and not $\mu$. A classical result here is that there is a bijection between the cosemistandard tableaux that have  shape $\mu$ and entries from $\{1,...,n\}$, with a basis of the $K$-vector space $L_{\mu}$ given by 
\begin{center}
	\begin{ytableau}
		i_1 & i_2 & \dots & i_b &\dots &i_a\\
		j_1 & j_2 & \dots & j_b
	\end{ytableau}
	$\mapsto  \pi_\mu(e_{i_1} e_{i_2} \dots e_{i_a} \otimes e_{j_1} e_{j_2} \dots e_{j_b}).$
\end{center} We refer to the elements of this basis of $L_\mu$ as \textit{cosemistandard basis elements}.
\section{Presentations of Schur modules}
We consider here the question which homomorphisms of $G$-modules $\Phi : \Lambda^a \otimes \Lambda^b \to \Lambda^a \otimes \Lambda^b$, where $a \ge b$, have cokernel an irreducible $G$-module, i.e. a module of the form $L_{(a+i, b-i)}$ for some $i\ge 0$. We will give an answer via the action of $\Phi$ on each irreducible summand of $\Lambda^a \otimes \Lambda^b$ and to this end we need to consider the actions of the various $\Phi_t$ on $\Lambda^a \otimes \Lambda^b$, that were defined in Definition \ref{maps}. 

\subsection{Two lemmas} We will need the relations of $L_{(a,b)}$ described in the following lemma, which is an analog of \cite[Lemma 4.2]{MS3} for Schur modules.
\begin{lem}\label{exch} Let $\alpha = (a, b)$ be a partition, $a_i, b_i$ nonnegative integers, where $i=1,2,3$, such that $a=a_1+a_2+a_3$ and $b=b_1+b_2+b_3$ and let $x \in \Lambda^{a_1+b_1}$, $y \in \Lambda^{a_2+b_2}$ and $z \in \Lambda^{a_3+b_3}$. Let  $\sum_u x_u(a_1)\otimes x'_u(b_1) \in \Lambda^{a_1} \otimes \Lambda^{b_1}$ and $ \sum_v y_v(a_2)\otimes y'_v(b_2)\in \Lambda^{a_2} \otimes \Lambda^{b_2}$ and $ \sum_w z_w(a_3)\otimes z'_w(b_3)\in \Lambda^{a_3} \otimes \Lambda^{b_3}$ be the images of $x,y,z$ respectively under the comultiplication maps $\Lambda^{a_i+b_i} \to $ $\Lambda^{a_i} \otimes \Lambda^{b_i}$, $i=1,2,3$, respectively. Consider \[\omega= \sum_{u,v,w}x_u(a_1)y_v(a_2)z_w(a_3) \otimes x'_u(b_1)y'_v(b_2)z'_w(b_3)\in \Lambda^{a} \otimes \Lambda^{b}.\] If $b_1 \le a_2+a_3$, then in $L_\alpha$ we have \begin{align*}\pi_{\alpha}(\omega) =&(-1)^{\epsilon} \sum_{i_2,i_3}(-1)^{\epsilon'}\tbinom{b_2+i_2}{b_2} \tbinom{b_3+i_3}{b_3}\\&\pi_{\alpha}\big(\sum_{v,w} xy_{v}(a_2-i_2)z_w(a_3-i_3) \otimes y'_v(b_2+i_2) z'_w(b_3+i_3)\big),\end{align*} where $\epsilon = (a_2+a_3+1)b_1$, $\epsilon'=i_2(a_3-i_3)+i_3b_2$ and the first sum is over all nonnegative integers $i_2, i_3$ such that $i_2+i_3=b_1$, $i_2 \le a_2$ and $i_3 \le a_3$.
	
\end{lem} \begin{proof}We argue by induction on $b_1$, the case $b_1=0$ being clear. Suppose $b_1 >0$. Let \[\omega'= \sum_{u,v,w}x_u(a_1+1)y_v(a_2)z_w(a_3) \otimes x'_u(b_1-1)y'_v(b_2)z'_w(b_3)\in \Lambda^{a+1} \otimes \Lambda^{b-1}.\] By Theorem \ref{cok1} we have $\pi_{\alpha} \circ \theta_1(\omega')=0$. Computing $\theta_1(\omega')$ using (\ref{comul}) we obtain \begin{equation*}\label{eq101} 0  = (-1)^ {a_2+a_3}b_1\pi_\alpha(\omega) + (-1)^{a_3+b_1-1}(b_2+1)\pi_\alpha(A) + (-1)^{b_1+b_2-1}(b_3+1)\pi_\alpha(B),\end{equation*} 
	where \begin{align*} &A=\sum_{u,v,w}x_u(a_1+1)y_v(a_2-1)z_w(a_3) \otimes x'_u(b_1-1)y'_v(b_2+1)z'_w(b_3),\\&
		B=\sum_{u,v,w}x_u(a_1+1)y_v(a_2)z_w(a_3-1) \otimes x'_u(b_1-1)y'_v(b_2)z'_w(b_3+1).	
	\end{align*}
	Hence $\pi_\alpha(\omega)$ is equal to \begin{equation}\label{eq102}  (-1)^ {a_2+a_3+1}\frac{1}{b_1}\big((-1)^{a_3+b_1-1}(b_2+1)\pi_\alpha(A) + (-1)^{b_1+b_2-1}(b_3+1)\pi_\alpha(B)\big),\end{equation} 
	By induction we have
	\begin{align*} \pi_{\alpha}(A) =&(-1)^{\zeta} \sum_{\substack{j_2+j_3=b_1-1\\ j_2 \le a_2-1, \; j_3 \le a_3}}(-1)^{\zeta_1}\tbinom{b_2+1+j_2}{b_2+1} \tbinom{b_3+j_3}{b_3}\\&\pi_{\alpha}\big(\sum_{v,w} xy_{v}(a_2-(1+j_2))z_w(a_3-j_3) \otimes y'_v(b_2+(1+j_2)) z'_w(b_3+j_3)\big),	
	\end{align*}
	\begin{align*} \pi_{\alpha}(B) =&(-1)^{\zeta} \sum_{\substack{j_2+j_3=b_1-1\\ j_2 \le a_2, \; j_3 \le a_3-1}}(-1)^{\zeta_2}\tbinom{b_2+i_2}{b_2} \tbinom{b_3+1+j_3}{b_3+1}\\&\pi_{\alpha}\big(\sum_{v,w} xy_{v}(a_2-j_2)z_w(a_3-(1+j_3)) \otimes y'_v(b_2+j_2) z'_w(b_3+(1+j_3))\big),	
	\end{align*}where $\zeta=(a_2+a_3)(b_1-1)$, $\zeta_1 = j_2(a_3-j_3)+j_3(b_2+1)$ and $\zeta_2 = j_2(a_3-1-j_3)+j_3b_2$.
	Substituting the above expressions in (\ref{eq101}) and computing the signs, we see  that the coefficient of 
	\[\pi_{\alpha}\big(\sum_{v,w} xy_{v}(a_2-i_2)z_w(a_3-i_3) \otimes y'_v(b_2+i_2)) z'_w(b_3+i_3)\big),\]
	where $i_2+i_3=b_1$, is equal to
	\[(-1)^{\epsilon +\epsilon'}\frac{1}{b_1}\big( (b_2+1)\tbinom{b_2+i_2}{b_2+1}\tbinom{b_3+i_3}{b_3} + (b_3+1)\tbinom{b_2+i_2}{b_2}\tbinom{b_3+i_3}{b_3+1}\big).\] 
	But this is equal to 
	$(-1)^{\epsilon +\epsilon'}\tbinom{b_2+i_2}{b_2}\tbinom{b_3+i_3}{b_3}$	
	and the desired result follows.		
\end{proof}

Each map $\Phi \in \Hom_G(\Lambda^a \otimes \Lambda^b, \Lambda^a \otimes \Lambda^b)$ acts as a scalar on every irreducible summand of $\Lambda^a \otimes \Lambda^b$. 	These summands are given by Pieri's rule \cite[2.2]{F}, $\Lambda^a \otimes \Lambda^b = \bigoplus_{i=0}^bL_{(a+i, b-i)}$. We compute the scalars in the next lemma when $\Phi=\Phi_t$. This is important for us as the $\Phi_t$, where $t\in \{0,\dots, b\}$, form a basis of the vector space $\Hom_G(\Lambda^a \otimes \Lambda^b, \Lambda^a \otimes \Lambda^b)$ (see Corollary \ref{basis} below).

\begin{lem}\label{lemmamain} Let $a,b,i,t \in \mathbb{N}$ such that $a \ge b \ge i$ and let $\mu=(a+i,b-i)$.
	
	\emph{(1)} Let $i < t$. Then the restriction of $\Phi_t$ to the summand $L_{\mu}$ of $\Lambda ^a \otimes \Lambda^b$ is equal to zero.
	
	\emph{(2)} Let $i \ge t$. Define $c_t(a,b,i)=\tbinom{i}{t}\tbinom{a-b+t+i}{t}$. Then the following diagram is commutative
	\begin{center}
		\begin{tikzcd}
			\Lambda^a \otimes \Lambda^b \arrow[r, "\Phi_t"] \arrow[d, "\delta_{i}"]  
			& \Lambda^{a} \otimes \Lambda^{b} \arrow[d, "\delta_{i}"]   \\
			\Lambda^{a+i} \otimes \Lambda^{b-i} \arrow[d, "\pi_{\mu}"]
			&  \Lambda^{a+i} \otimes \Lambda^{b-i} \arrow[d, "\pi_{\mu}"] \\
			L_{\mu} \arrow[r, "c_{t}{(a,b,i)}"'] & L_{\mu}
		\end{tikzcd}
	\end{center} where the bottom horizontal map is multiplication by $c_t(a,b,i)$. Moreover, if $n \ge a+b$, then the map $\pi_{\mu} \circ \delta_i : \Lambda ^{a} \otimes \Lambda ^b \to L_{\mu}$ is nonzero.
\end{lem}
\begin{proof} (1) By definition, $\Phi_t$ factors through $\Lambda^{a+t}\otimes\Lambda^{b-t}$. But the irreducible module $L_{\mu}$ is not a summand of 	$\Lambda^{a+t}\otimes\Lambda^{b-t}$	if $i<t$, according to Pieri's rule.
	
	(2) For the first statement, it suffices to show that in the following diagram the left and right rectangles commute
	\begin{center}
		\begin{tikzcd}
			\Lambda^a \otimes \Lambda^b \arrow[r, "\delta_t"] \arrow[d, "\delta_{i}"]  
			& \Lambda^{a+t} \otimes \Lambda^{b-t} \arrow[d, "\delta_{i-t}"] \arrow[r, "\theta_{t}"]& \Lambda^{a} \otimes \Lambda^{b} \arrow[d, "\delta_{i}"]   \\
			\Lambda^{a+i} \otimes \Lambda^{b-i} \arrow[d, "\pi_{\mu}"]
			& \Lambda^{a+i} \otimes \Lambda^{b-i} \arrow[d, "\pi_{\mu}"]&  \Lambda^{a+i} \otimes \Lambda^{b-i} \arrow[d, "\pi_{\mu}"] \\
			L_{\mu} \arrow[r, "\tbinom{i}{t}"'] &L_{\mu} \arrow[r, "\tbinom{a-b+t+i}{t}"']& L_{\mu}
		\end{tikzcd}
	\end{center}
	The left rectangle commutes by the first equality of	(\ref{comm}).
	
	Let $x\otimes y \in \Lambda^{a+t} \otimes \Lambda^{b-t}$. In the clockwise direction we have \begin{align*}&\pi_{\mu} \circ \delta_i \circ \theta_t(x\otimes y) =\pi_{\mu} \circ \delta_i\big(\sum_u x_u(a)\otimes x'_u(t)y\big) \\&= \pi_{\mu}\big( \sum_{i_1, i_2}(-1)^{(t-i_1)i_2}\tbinom{a+i_1}{i_1}\sum_{u,v}x_u(a+i_1)y_v(i_2)\otimes x'_u(t-i_1)y'_v(b-t-i_2)\big),\end{align*} where we applied (\ref{comul}) in the second equality. The second sum is over all $i_1,i_2 \in \mathbb{N}$ such that $i_1+i_2=i$ and $i_1 \le t$. We apply Lemma \ref{exch} to the third sum to obtain $\pi_{\mu} \circ \delta_i \circ \theta_t(x\otimes y)$ is equal to
	\begin{align*}& \pi_{\mu}\big( \sum_{i_1, i_2}(-1)^{(t-i_1)i_2}\tbinom{a+i_1}{i_1}(-1)^{(i_2+1)(t-i_1)}\tbinom{b-i}{t-i_1}\sum_{v}xy_v(i-t)\otimes y'_v(b-i)\big) \\&=\Big(\sum_{i_1=0}^{t}(-1)^{t-i_1}\tbinom{a+i_1}{i_1}\tbinom{b-i}{t-i_1}\Big)\pi_{\mu}(\sum_{v}xy_v(i-t)\otimes y'_v(b-i)).\end{align*} By the first identity of \cite[Lemma 4.1(2)]{MS3}, we have  $\sum_{i_1=0}^{t}(-1)^{t-i_1}\tbinom{a+i_1}{i_1}\tbinom{b-i}{t-i_1}=$ $ \tbinom{a-b+t+i}{t}$. Hence $\pi_{\mu} \circ \delta_i \circ \theta_t(x\otimes y)$ is equal to \begin{equation*}\tbinom{a-b+t+i}{t}\pi_{\mu}(\sum_{v}xy_v(i-t)\otimes y'_v(b-i))=\tbinom{a-b+t+i}{t}\pi_{\mu} \circ \delta_{i-t}(x\otimes y) \end{equation*} and the right rectangle commutes.
	
	Since $n \ge a+b$, we have the element $x=e_1\cdots e_a\otimes e_{a+1} \cdots e_{a+b} \in \Lambda^a \otimes \Lambda^b$. From the definition of $\delta_{i}$ it follows that  $\pi_{\mu} \circ \delta_{i}(x)$ is a sum of elements of the form $
	\pm\pi_{\mu}(e_1 \cdots e_a u_1 \cdots u_{i} \otimes u_{i+1} \cdots u_b),$
	where $u_1,\dots,u_b \in \{e_{a+1},\dots,e_b\}$ and $\{u_1,\dots,u_{i}\} \cap$ $\{u_{i+1},\dots,u_{b}\} = \emptyset $.	The number of the $u_{i+1},\dots,u_{b}$ is equal to $b-i$ and we have $b-i \le a. $ Thus there are no violations in the columns, meaning that the above elements  are (up to sign) cosemistandard basis elements of $L_{\mu}$. They are clearly distinct. Hence we conclude that $\pi_\mu \circ \delta_{i}(x)$ is nonzero and the map $\pi_\mu \circ \delta_{i}$ is nonzero.	\end{proof} 

\begin{cor}\label{basis}
	Let $a,b \in \mathbb{N}$ such that $a \ge b$. If $n \ge a+b$, then a basis of the vector space $\Hom_{G}(\Lambda^a \otimes \Lambda^b, \Lambda^a \otimes \Lambda^b )$ is $\{\Phi_0,\Phi_1,\dots,\Phi_b\}$.
\end{cor}

\begin{proof} 	By Pieri's rule \cite[2.2]{F}, $\Lambda^a \otimes \Lambda^b = \bigoplus_{i=0}^bL_{(a+i, b-i)}$. Since the $G$-modules $L_{(a+i, b-i)}$ are irreducible and distinct, we have $\dim\Hom_{G}(\Lambda^a \otimes \Lambda^b, \Lambda^a \otimes \Lambda^b )=b+1.$ By Lemma \ref{lemmamain} it follows that $\Ima(\Phi_t)=$ $\bigoplus_{i \ge t}L_{(a+i,b-i)},$ $ t=0,\dots,b.$ Hence the $\Phi_0,\dots,\Phi_b$ are linearly independent. 
\end{proof}
\begin{rem}\label{subset}
	For later use we record that $\Ima(\Phi_i) \subseteq \Ima(\Phi_1)\subseteq \Ima(\theta_1)$ if $i>0$.
\end{rem}

\subsection{Main result}	Our main result for the case of two rows is the following, which provides an arithmetic criterion when the cokernel of $\Phi \in \Hom_{G}(\Lambda^a \otimes \Lambda^b, \Lambda^a \otimes \Lambda^b )$ is irreducible. Recall the notation $c_t(a,b,i)=\tbinom{i}{t}\tbinom{a-b+t+i}{t}$ from Lemma \ref{lemmamain}.\begin{thm}\label{main}
	Let $a,b \in \mathbb{N}$ such that $a \ge b$ and let $n \ge a+b$. Let $\Phi \in \Hom_{G}(\Lambda^a \otimes \Lambda^b, \Lambda^a \otimes \Lambda^b )$ and write $\Phi=\sum_{t=0}^{b}a_t\Phi_t$, where $a_t \in K$. Then the cokernel of $\Phi$ is isomorphic to $L_{(a+i,b-i)}$ if and only if \begin{enumerate}
		\item $\sum_{t=0}^{i}a_tc_t(a,b,i)=0$, and 
		\item $\sum_{t=0}^{j}a_tc_t(a,b,j) \neq0$ for all $j \in \{0,\dots,b\} \setminus \{i\} $.
	\end{enumerate}
\end{thm} \begin{proof}
	We have $\Lambda^a \otimes \Lambda^b = \bigoplus_{j=0}^bL_{(a+j, b-j)}$. By Lemma \ref{lemmamain}, the map $\Phi$ acts on $L_{(a+j, b-j)}$ as multiplication by $\sum_{t=0}^{j}a_tc_t(a,b,j)$. The result follows.
\end{proof}
The previous result for $i=0$ may be generalized easily to the case of partitions $\mu =(\mu_1, \dots, \mu_{\ell} )$ consisting of ${\ell} \ge 2$  parts. We need some notation. For $c \in \{1, \dots , {\ell}-1\}$, let $\mu(c)$ be the partition $(\mu_c, \mu_{c+1})$. If $\Psi_{\mu(c)} \in \Hom_G(\Lambda^{\mu(c)},\Lambda^{\mu(c)})$, let $\overline{\Psi}_{\mu(c)} \in \Hom_G(\Lambda^{\mu},\Lambda^{\mu})$ be defined by $\overline{\Psi}_{\mu(c)}=1 \otimes \cdots \otimes 1 \otimes \Psi_{\mu(c)} \otimes 1 \otimes \cdots \otimes 1.$ Let $\{\Phi_{\mu(c),0}, \dots, \Phi_{\mu(c),\mu_{c+1}} \}$ be the basis of $\Hom_G(\Lambda^{\mu(c)},\Lambda^{\mu(c)})$ given in Corollary \ref{basis}. There are uniquely determined $a_{\mu(c), t} \in K$ such that $\Psi_{\mu(c)} = \sum_{t=0}^{\mu_{c+1}}a_{\mu(c), t} \Phi_{\mu(c), t}$.  Our main result for many rows is the following. We use the notation $|\mu| = \mu_1 + \dots +\mu_\ell$.
\begin{cor}\label{maincor}With the previous notation, suppose $n \ge |\mu|$. Then the cokernel of $ \sum_{c=1}^{{\ell}-1}\overline{\Psi}_{\mu(c)} :\bigoplus_{c=1}^{{\ell}-1} \Lambda^{\mu} \to \Lambda^{\mu}$ is isomorphic to $L_{\mu}$ if  \begin{center}
		$a_{\mu(c),0} =0$ and $\sum_{t=1}^{j}a_{\mu (c),t}c_t(\mu_c,\mu_{c+1},j) \neq0$
	\end{center} for all $c=1,\dots,{\ell}-1$ and $j=1,\dots,\mu_{c+1}$.\end{cor}
\begin{proof} From Remark \ref{subset} it follows that $\Ima(\sum_{c=1}^{{\ell}-1}\overline{\Psi}_{\mu(c)}) \subseteq \Ima\theta_\mu$, where $\theta_\mu$ is the map in Theorem \ref{cok1}. Hence $ \coker(\sum_{c=1}^{{\ell}-1}\overline{\Psi}_{\mu(c)}) \simeq L_\mu$ if and only if 
	$\Ima(\sum_{c=1}^{{\ell}-1}\overline{\Psi}_{\mu(c)}) = \Ima\theta_\mu$. 
	
	Again by Remark \ref{subset}, \begin{equation}\label{sub}\Ima({\Psi}_{\mu(c)}) \subseteq \Ima(\Lambda^{\mu_c +1} \otimes \Lambda^{\mu_{c+1} -1} \xrightarrow{\theta_1}\Lambda^{\mu_c}  \otimes \Lambda ^{\mu_{c+1}}).\end{equation} Under the hypothesis of the corollary, we apply  Theorem \ref{main} to  $\coker ({\Psi}_{\mu(c)})$ and we conclude that in (\ref{sub}) we have equality.  (Notice that for $i=0$ in Condition 1 of Theorem \ref{main}, we have $\sum_{t=0}^{i}a_tc_t(a,b,i)$ $=a_0$.) Hence $\Ima(\sum_{c=1}^{{\ell}-1}\overline{\Psi}_{\mu(c)}) = \Ima\theta_\mu$.
\end{proof}

\section{Schur modules and symmetrization of Garnir relations} Here we apply Corollary \ref{maincor} to particular maps $\psi_k \in \Hom_G(\Lambda^a \otimes \Lambda^b, \Lambda^a \otimes \Lambda^b)$ defined below that have an interesting combinatorial interpretation; they are obtained from the Garnir relations by a symmetrization process. As we mentioned in the Introduction, such  symmetrizations was first considered in \cite{BF} for $k=1$.

Throughout this section, let $a\ge b \ge k $ be positive integers.
\begin{defn}\label{maps2} \;
	\begin{enumerate} \item Define $\phi_k \in \Hom _G(\Lambda^a \otimes \Lambda^b, \Lambda^a \otimes \Lambda^b)$ as the composition
		\begin{align*}\Lambda^a \otimes \Lambda^b &\xrightarrow{\Delta_{a-k,k} \otimes \Delta_{k,b-k}} \Lambda^{a-k} \otimes \Lambda ^k \otimes \Lambda ^k \otimes \Lambda ^{b-k}\\&\xrightarrow{1 \otimes \tau \otimes 1} \Lambda^{a-k} \otimes \Lambda ^k \otimes \Lambda ^k \otimes \Lambda ^{b-k}\xrightarrow{m \otimes m}\Lambda^{a}  \otimes \Lambda ^{b},\end{align*}
		where $\tau : \Lambda^k \otimes \Lambda^k \to \Lambda^k \otimes \Lambda^k, \tau (w \otimes z)=z \otimes w$. 
		\item Define $\psi_k \in \Hom _G(\Lambda^a \otimes \Lambda^b, \Lambda^a \otimes \Lambda^b)$ by $\psi_k=\tbinom{b}{k}1_{\Lambda^a \otimes \Lambda ^b}-\phi_k.$ \end{enumerate}\end{defn}

\begin{rem}\label{ex}It is straightforward to verify that if $x\otimes y= e_{i_1} \cdots e_{i_a} \otimes e_{j_1} \cdots e_{j_b} $ is  a basis element of $\Lambda^a \otimes \Lambda^b$, then $\phi_k(x \otimes y)$ is obtained by exchanging any $k$ elements $x_1 <x_2 < \dots < x_k$ of  $\{e_{i_1}, \dots,  e_{i_a}\}$ with any $k$ elements $y_1<y_2 < \dots < y_k $ of $\{e_{j_1}, \dots , e_{j_b}\}$ and putting the $x_1, \dots, x_k$ (respectively, $y_1, \dots, y_k$) in the original positions  of the $y_1, \dots , y_k$ (respectively, $x_1, \dots, x_k$).\end{rem}

In terms of the basis of Corollary \ref{basis}, we have the following expressions.
\begin{lem}\label{exp}
	With the previous notation, \[
	\phi_k = \sum_{t=0}^{k}(-1)^{t}\tbinom{b-t}{k-t}\Phi_t \;\; \text{and} \;\;
	\psi_k = \sum_{t=1}^{k}(-1)^{t-1}\tbinom{b-t}{k-t}\Phi_t.\]
\end{lem}
\begin{proof} Let $ x \otimes y \in \Lambda^a \otimes \Lambda ^b$. From the definitions we have that $	\Phi_t(x\otimes y)=$ $ \theta_t\circ \delta _t (x \otimes y)=$ $\theta_t\big(\sum_vxy_v(t) \otimes y'_v(b-t)\big)$. Using  (\ref{comul}) and (\ref{comm}),  this is equal to \begin{equation*}
		\sum_{i=0}^{t}(-1)^{\epsilon}\tbinom{b-t+i}{i}\sum_{u,v}x_u(a-t+i)y_v(t-i) \otimes x'_u(t-i)y'_v(b-t+i),	\end{equation*} where $\epsilon = (t-i)^2$. Hence \begin{equation}\label{matrix}
		\Phi_t=\sum_{i=0}^{t}(-1)^{t-i}\tbinom{b-t+i}{i}\phi_{t-i}=\sum_{k=0}^{t}(-1)^{k}\tbinom{b-k}{b-t}\phi_{k},
	\end{equation}where in the last equality we set $k=t-i$. Consider the $(b+1) \times (b+1)$ matrix $A=(a_{kt})$, where $a_{kt}=(-1)^{k}\tbinom{b-k}{b-t}$ and $0 \le k,t \le b$. From the binomial identity $\sum_{j=0}^{b}(-1)^{k+j} \tbinom{b-k}{b-j} \tbinom{b-j}{b-t}=\delta_{k,t}$, see \cite[p. 4]{R}, where $\delta_{k,t}$ is Kronecker's delta,  it follows that $A^2=I$, the identity matrix. Hence we may invert (\ref{matrix}) to obtain $
	\phi_k = \sum_{t=0}^{k}(-1)^{t}\tbinom{b-t}{b-k}\Phi_t$ as desired. 
	
	The second equality follows from the first and Definition \ref{maps2}.
\end{proof}
From Lemma \ref{exp} and Theorem \ref{main} (for $i=0$) we obtain:
\begin{cor}\label{2rowscr}
	Let  $n \ge a+b$. Then $\coker(\psi_k) \simeq L_{(a,b)}$ if and only if \[\sum_{t=1}^{j}(-1)^{t-1}\tbinom{b-t}{b-k}\tbinom{j}{t}\tbinom{a-b+j+t}{t} \neq0\] for all $j=1,\dots,b$. 
\end{cor}
\section{Schur modules and Garnir relations} Our main goal here is to obtain an analog of Corollary \ref{2rowscr} for maps $\gamma_k$ (in place of $\psi_k$) that correspond to Garnir relations. This is accomplished in Corollary \ref{cor2rowsGR} below. Except for Lemma 3.1, the present section is independent of the previous two. However, the main ideas and techniques are similar.

Throughout this section, let $a\ge b \ge k $ be positive integers.

\begin{defn}\label{mapsG} \;
	\begin{enumerate} \item Define $\beta_k \in \Hom _G(\Lambda^a \otimes \Lambda^{k} \otimes \Lambda^{b-k}, \Lambda^a \otimes \Lambda^b)$ as the composition
		\begin{align*}\Lambda^a \otimes \Lambda^k \otimes \Lambda^{b-k} &\xrightarrow{\Delta_{a-k,k} \otimes 1} \Lambda^{a-k} \otimes \Lambda ^k \otimes \Lambda ^k \otimes \Lambda ^{b-k}\\&\xrightarrow{1 \otimes \tau \otimes 1} \Lambda^{a-k} \otimes \Lambda ^k \otimes \Lambda ^k \otimes \Lambda ^{b-k}\xrightarrow{m \otimes m}\Lambda^{a}  \otimes \Lambda ^{b},\end{align*} \item Define $\gamma_k \in \Hom _G(\Lambda^a \otimes \Lambda^{k} \otimes \Lambda^{b-k}, \Lambda^a \otimes \Lambda^b)$
		by \[\gamma_k(x\otimes y\otimes z)= x \otimes yz-\beta_k(x\otimes y\otimes z).\] \end{enumerate}
\end{defn}

We will need the following lemma.
\begin{lem}\label{Garnirsubset} We have $\Ima (\gamma_k) \subseteq \Ima (\theta_1)$, where $\theta_1 :\Lambda^{a+1}\otimes\Lambda^{b-1} \to \Lambda^{a}\otimes\Lambda^{b}$ is given in Definition \ref{maps}.
\end{lem}
\begin{proof}
	Consider the composition \[\Lambda^a \otimes \Lambda^k \otimes \Lambda^{b-k}\xrightarrow{\gamma_k}\Lambda^{a} \otimes \Lambda^{b}\xrightarrow{\pi_\mu} L_\mu\]
	where $\mu=(a,b)$, and let $x \otimes y \otimes z \in \Lambda^a \otimes \Lambda ^k \otimes \Lambda^{b-k}$. We have \[\pi_\mu \circ  \gamma_k(x \otimes y \otimes z)=\pi_\mu(x\otimes yz)-\big(\sum_ux'_u(a-k)y\otimes x'_u(k)z\big)=0,\]
	where the second equality follows from Lemma \ref{exch}. Hence $\pi_\mu \circ  \gamma_k=0.$ By Theorem \ref{cok1} we have $\Ima (\gamma_k) \subseteq \ker(\pi_\mu) = \Ima(\theta_1).$
\end{proof}

For the proof of the next lemma we will need the following 
remark. Recall we have $a \ge b \ge k \ge 1$. Let $b \ge i \ge 0$ and define \[B(a,b,k,i)=\{(i_1,i_2) \in \mathbb{N} \times \mathbb{N}: i_1+i_2 = i, \; i_1 \le k, \; i_2 \le b-k\}.\]
It is straightforward to verify that \begin{equation}\label{rempq}
	B(a,b,k,i)=\{(t, i-t) \in \mathbb{N} \times \mathbb{N}: p \le t \le q\},
\end{equation} where $p=i-\min\{b-k,i\}$ and $q=\min\{k,i\}$.

\begin{lem}\label{thmmain22} Suppose $n \ge a+b.$ Define \[d_k(a,b,t)=1-\sum_{j=p}^{t}(-1)^{j}\tbinom{a-k+j}{a-k} \tbinom{b-k-i+t}{t-j}\] for $t=p, p+1, \dots, q$. Then $\coker(\gamma_k) \simeq L_{(a,b)}$ if and only if for every $i$ such that $1 \le i \le b$, there exists $t$ such that $p \le t \le q$ and $ d_k(a,b,t)\neq0$. 
\end{lem}	

\begin{proof}
	
	Consider the composition	\[\Lambda^a \otimes \Lambda^k \otimes \Lambda^{b-k}\xrightarrow{\gamma_k}\Lambda^{a} \otimes \Lambda^{b}\xrightarrow{\delta_i}\Lambda^{a+i} \otimes \Lambda^{b-i}\xrightarrow{\pi_\mu} L_\mu,\]
	where $\mu=(a+i,b-i)$, and let $x \otimes y \otimes z \in \Lambda^a \otimes \Lambda ^k \otimes \Lambda^{b-k}$. A quick computation using the definitions of the maps (\ref{comul}) and (\ref{rempq}) yields \begin{equation}\label{42}\pi_\mu \circ \delta_i \circ \gamma_k (x\otimes y \otimes z) =  \sum_{j=p}^q \eta_j T(j)-\sum_{j=p}^q \epsilon_j\tbinom{a-k+j}{a-k}A(j),\end{equation}
	where $\eta_t=(-1)^{(k-t)(i-t)}$, $\epsilon_j=(-1)^{jp+(i-j)(k-j)}$ and
	\begin{align*}
		&T(j)=\pi_\mu \big( \sum_{v,w}xy_v(j)z_w(i-j)\otimes y'_v(k-j)z'_w(b-k-i+j)\big),\\
		&A(j)=\pi_\mu \big( \sum_{u,w}x_u(a-k+j)yz_w(i-j)\otimes x'_u(k-j)z'_w(b-k-i+j)\big).
	\end{align*}
	We apply Lemma \ref{exch} to $A(j)$ to obtain 
	\begin{equation*}
		A(j)=\zeta_j \sum_{t=j}^{q}\eta_t \tbinom{b-k-i+t}{t-j} T(t).
	\end{equation*}
	Substituting this in (\ref{42}) and noting that $\epsilon_j \zeta_j =(-1)^{j}$, we conclude that  
	\begin{equation}\label{43}\pi_\mu \circ \delta_i \circ \gamma_k (x\otimes y \otimes z) =  \sum_{t=p}^q d_k(a,b,t) \eta_t T(t).\end{equation}
	
	Now since $n \ge a+b$, we may consider the choice $x \otimes y \otimes z = e_1 \cdots e_a \otimes e_{a+1} \cdots e_{a+k} \otimes e_{a+k+1} \cdots e_{a+b} \in \Lambda^a \otimes \Lambda ^k \otimes \Lambda^{b-k}$. Since $a \ge b-i$, it follows from the definition of $T(j)$ above, that this choice has the property that $T(p), \dots , T(q)$ are distinct cosemistandard basis elements of $L_{(a+i,b-i)}$ (see the last paragraph of Section 2). By Pieri's rule we know that $\Lambda^a \otimes \Lambda^b$ contains a unique copy of $L_{(a+i,b-i)}$, for each $1 \le i \le b$. Hence from (\ref{43}) we conclude that the image $\Ima (\gamma_k)$ of the map $\gamma_k$ intersects non trivially the copy of $L_{(a+i,b-i)}$ inside $\Lambda^a \otimes \Lambda^b$ if an only if one of the integers $d_k(a,b,t) \eta_t$, where $t=p, \dots, q$, is nonzero. Since $\eta_t = \pm1$,   this is equivalent to one of the integers $d_k(a,b,t)$ being nonzero.
	
	In order to complete the proof, it remains to be shown that $\Ima (\gamma_k)$ does not intersect the (unique) copy of $L_{(a,b)}$ inside $\Lambda^a \otimes \Lambda^b$. Consider the composition \[\Lambda^a \otimes \Lambda^k \otimes \Lambda^{b-k}\xrightarrow{\gamma_k}\Lambda^{a} \otimes \Lambda^{b}\xrightarrow{\pi_{(a,b)}} L_{(a,b)}.\]
	In the proof of Lemma \ref{Garnirsubset} we saw that $\pi_{(a,b)} \circ  \gamma_k=0.$
\end{proof}

Next we observe that we may obtain a better version of the previous lemma, meaning that we may avoid the use of the expressions $d_k(a,b,t)$, if we are interested not on the action on irreducible summands but only when  $\coker(\gamma_k) \simeq L_{(a,b)}$ holds.
\begin{cor}\label{cor2rowsGR}
	Suppose $n \ge a+b$. Then $\coker(\gamma_k) \simeq L_{(a,b)}$ if and only if $a>k$ or  $a=b=k=1$.
\end{cor}
\begin{proof} With the notation of the previous lemma, let $1 \le i \le b$. 
	
	(1) Suppose $a >k$. Let $p \neq 0$.	By definition, $d_k(a,b,p)=1-(-1)^{p}\tbinom{a-k+p}{a-k}$. Since $a > k$ we have $\tbinom{a-k+p}{a-k} \neq 1$ and hence $d_k(a,b,p) \neq 0$.
	
	Let $p =0$. Since $q=\min\{k,i\} \ge 1$, we may consider $d_k(a,b,t)$ for $t=1$. By definition, $d_k(a,b,1)=1-\tbinom{a-k}{a-k}\tbinom{b-k-i+1}{1}+\tbinom{a-k+1}{a-k}\tbinom{b-k-i+1}{0}=1+a-b+i.$ Since $a\ge b$, we have $d_k(a,b,1) \neq 0.$
	
	(2) Suppose $a=b=k=1$. Then $a=b=k=i=p=q=t=1$ and $d_1(a,b,1)=1-(-a)=2>0$.
	
	So for any $1 \le i \le b$ we have found a nonzero $d_k(a,b,t)$ for some $p \le t \le q$. By Lemma \ref{thmmain22}, $\coker(\gamma_k) \simeq L_{(a,b)}$.
	
	Conversely, suppose $a=b=k>1$. We will show that $\coker(\gamma_k)$ and $ L_{(a,b)}$ are not isomorphic. Indeed, we have $p=q=i$ and thus $t$ is uniquely determined, $t=i$. Choosing $i=2$, we have $d_k(a,b,t)=1-(-1)^2=0$. By Lemma \ref{thmmain22}, $\coker(\gamma_k)$ and  $L_{(a,b)}$ are not isomorphic. \end{proof}

We note that according to the previous corollary, we have $\coker(\gamma_k) \simeq L_{(a,b)}$ for all $a \ge b \ge k \ge 1$ except when $a=b=k>1$. In this exceptional case, we have $\coker(\gamma_k) \simeq S_2(\Lambda^k)$, the second symmetric power of $\Lambda^k$.

In order to state the main result of this section, that involves partitions with more parts, we need some notation. Let $\mu=(\mu_1, \dots, \mu_\ell)$ be a partition of length $\ell \ge 2$ and let $k_1,\dots, k_{\ell-1}$ positive integers satisfying $k_c \le \mu_{c+1}$, $c=1,\dots,\ell-1$.
Let $\mu(c)$ be the partition $(\mu_c, \mu_{c+1})$ and let \[\gamma_{k_c} \in \Hom _G(\Lambda^{\mu_c} \otimes \Lambda^{k_c} \otimes \Lambda^{\mu_{c+1}-k_c}, \Lambda^{\mu_c} \otimes \Lambda^{\mu_{c+1}})\] be the map of Definition \ref{mapsG}(2). Define \[\overline{\gamma}_{k_c} \in \Hom_G(\Lambda^{\mu_1} \otimes \cdots \otimes \Lambda^{\mu_c} \otimes \Lambda^{k_c} \otimes \Lambda^{\mu_{c+1}-k_c} \otimes \cdots \otimes \Lambda^{\mu_\ell},\Lambda^{\mu})\]  by $\overline{\gamma}_{k_c}=1 \otimes \cdots \otimes 1 \otimes \gamma_{k_c} \otimes 1 \otimes \cdots \otimes 1.$

\begin{thm}\label{maincorGarnir3} Let $\mu=(\mu_1, \dots, \mu_\ell)$ be a partition of length $\ell \ge 2$ and let $k_1,\dots, k_{\ell-1}$ positive integers satisfying $k_c \le \mu_{c+1}$, $c=1,\dots,\ell-1$. Suppose $n \ge |\mu|$. If $\mu_c > \mu_{c+1}$ for those $c$ such that $k_c=\mu_{c+1}>1$, then the cokernel of \[ \sum_{c=1}^{{\ell}-1}\overline{\gamma}_{k_c} :\bigoplus_{c=1}^{{\ell}-1} \Lambda^{\mu_1} \otimes \cdots \otimes \Lambda^{\mu_c} \otimes \Lambda^{k_c} \otimes \Lambda^{\mu_{c+1}-k_c} \otimes \cdots \otimes \Lambda^{\mu_\ell} \to \Lambda^{\mu}\] is isomorphic to $L_{\mu}$.\end{thm}
\begin{proof}From Lemma \ref{Garnirsubset} we have \begin{equation}\label{imsub}\Ima( \overline{\gamma}_{k_c}) \subseteq \Ima(\Lambda^{\mu_1} \otimes \cdots \otimes \Lambda^{\mu_c+1}\otimes \Lambda^{\mu_{c+1}-1} \otimes \cdots \otimes \Lambda^{\mu_{\ell}} \xrightarrow{1\otimes \cdots \otimes\theta_1 \otimes \cdots \otimes 1} \Lambda^{\mu}).\end{equation}
	Under the assumptions of the theorem, Corollary \ref{cor2rowsGR} says that $\coker(\overline{\gamma}_{k_c}) \simeq \coker (1\otimes \cdots \otimes\theta_1 \otimes \cdots \otimes 1)$. Hence in (\ref{imsub}) we have equality for all $c$. Thus \[\coker\big( \sum_{c=1}^{{\ell}-1}\overline{\gamma}_{k_c}\big) \simeq \coker(\theta_\mu) \simeq L_\mu,\]
	where the last isomorphism is due to Theorem \ref{cok1}.\end{proof}

\section{Presentations of Specht modules}
In this section we consider consequences of our previous results for Specht modules and discuss the relationship with previous known results, especially \cite{BF} and \cite{FHW}. 

Suppose $n \ge r$. The Schur functor is a functor $f$ from the category $M_K(n,r)$ of homogeneous polynomial representations of  $G$ of degree $r$ to the category $\mod \mathfrak{S}_r$ of left modules over the symmetric group $\mathfrak{S}_r$. For $M \in M_K(n,r)$,  $f(M)$ is the weight subspace $M_\alpha$ of $M$, where $\alpha = (1^r, 0^{n-r})$; for $\theta :M \to M'$ a morphism in $M_K(n,r)$, $f(\theta)$  is the restriction $M_\alpha \to M'_\alpha$ of $\theta$. We recall the following well known facts \cite{Gr}. Let $\mu$ be a partition of $r$ and $\lambda = \mu'$.  \begin{enumerate}
	\item $f$ is an exact functor, 
	\item $f(L_{\mu})=S^{\lambda}$ and $f(\Lambda^{\mu}) = M^{\lambda}$,  
	\item $f$ induces an isomorphism $\Hom_G(\Lambda^{\mu}, \Lambda^{\mu}) \simeq \Hom_{\mathfrak{S}_r}(M^{\lambda}, M^{\lambda})$.
\end{enumerate}
In fact, since the characteristic of $K$ is zero, $f$ is an equivalence of categories. 

\subsection{Specht modules and symmetrizations of Garnir relations}Consider the setup and notation of the first two paragraphs of Section \ref{results}. It follows from the above, Remark \ref{ex} and (\ref{newrel}) that when we apply the Schur functor $f$ to $\psi_k \in \Hom _G(\Lambda^a \otimes \Lambda^b, \Lambda^a \otimes \Lambda^b)$ we obtain the map $f(\psi_k) \in \Hom_{\mathfrak{S}_{a+b}}(M^{(a,b)'}, M^{(a,b)'})$ that satisfies $f(\psi_k)(\overline{T})=r^{T}_{1,k}$ for all $\overline{T} \in M^{(a,b)'}$. Hence the image is \begin{equation}\label{psiSGR}\Ima (f(\psi_k)) = SGR^{(a,b)'}(k).\end{equation}Thus from Corollary \ref{2rowscr} we conclude the following.
\begin{cor}\label{rowsSp}
	Let $a,b,k \in \mathbb{N}$ such that $a \ge b \ge k \ge 1$. Then as $\mathfrak{S}_{a+b}$-modules $M^{(a,b)'} /SGR^{(a,b)'}(k) \simeq S^{(a,b)'}$ if and only if  $\sum_{t=1}^{j}(-1)^{t-1}\tbinom{b-t}{b-k}\tbinom{j}{t}\tbinom{a-b+j+t}{t} \neq0$ for all $j=1,\dots,b$. 
\end{cor}
The 'if' direction of the above corollary generalizes to any partition $\lambda$ of $r$ in place of $(a,b)'$. We consider this as our second main result.
\begin{thm}\label{mainSpe}
	Let $\lambda$ be a partition of $r$ and $k_1,\dots, k_{\lambda_1-1}$ positive integers satisfying $k_c \le \mu_{c+1}$, $c=1,\dots,\lambda_1-1$, where $\mu = \lambda'$. Then as $\mathfrak{S}_r$-modules we have $M^{\lambda} / SGR^\lambda(k_1,\dots,k_{\lambda_1-1}) \simeq S^\lambda$ if \[\sum_{t=1}^{j}(-1)^{t-1}\tbinom{\mu_{c+1}-t}{\mu_{c+1}-k_c}\tbinom{j}{t}\tbinom{\mu_{c}-\mu_{c+1}+j+t}{t} \neq0\] for all $c=1,\dots,\lambda_1-1$ and  $j=1,\dots,\mu_{c+1}$.
\end{thm}
\begin{proof} This follows by applying the Schur functor to Corollary \ref{maincor} for the choice $\Psi_{\mu(c)} = \psi_{k_c} \in \Hom_G(\Lambda^{\mu_c} \otimes \Lambda^{\mu_{c+1}},  \Lambda^{\mu_c} \otimes \Lambda^{\mu_{c+1}})$, $c=1,\dots,\lambda_1 -1$. From (\ref{psiSGR}) we have $\Ima (f(\sum_{c=1}^{{\ell}-1}\overline{\psi}_{k_c})) = SGR^{\lambda}(k_1, \dots, k_{\lambda_1-1})$. We know from the second statement of Lemma \ref{exp} that $a_{\mu (c),t} = (-1)^{t-1}\tbinom{\mu_{c+1}-t}{\mu_{c+1}-k_c}$ if $t \neq 0$ and $a_{\mu (c),t} =0$ if $t=0$.
\end{proof}
\begin{exs}\label{exs} For short, let \[\Sigma(k,j)=\sum_{t=1}^{j}(-1)^{t-1}\tbinom{b-t}{b-k}\tbinom{j}{t}\tbinom{a-b+j+t}{t}\] be the sum in Corollary \ref{rowsSp},  where $1 \le j \le b$. 
	
	(1) Let $k=1$. Since $\tbinom{b-t}{b-k}=0$ if $t>1$, we have  $\Sigma(1,j)=j(a-b+j+1)$ for every $j=1, \dots, b$. Since $a\ge b$, we have $\Sigma(1,j) \neq 0$ . From Theorem \ref{mainSpe} we conclude that for any partition $\lambda$ we have 
	\[M^{\lambda} / SGR^\lambda(1,\dots,1) \simeq S^\lambda.\]
	Thus we recover  Theorem 3.5  of \cite{BF} (which we have stated as Theorem \ref{BF} in the Introduction).
	
	(2) Let $k=b$ the maximal value of $k$. First we note that in this case there is no symmetrization, i.e. we have $SGR^{(a,b)'}{(b)} =GR^{(a,b)'}(b)$. Using the first identity of \cite[Lemma 4.1(2)]{MS3}, we conclude that $\Sigma(b,j) =-(-1)^j \tbinom{a-b+j}{j}+1.$
	Taking into account that $j \neq 0$, we see that  the left hand side is equal to $0$ if and only if $j>0$ is even and $a=b$. Since $j\le b$, it follows from Corollary \ref{rowsSp} that  $M^{(a,b)'}/ GR^{(a,b)'}(b) =S^{(a,b)'}$ if and only if $b=1$ or $b>1$ and $a \neq b$. Hence from Theorem \ref{mainSpe} we conclude that \[M^{\lambda} / GR^\lambda(\lambda'_2,\dots,\lambda'_{\lambda_1}) \simeq S^\lambda,\]
	for any partition $\lambda$ satisfying $\lambda'_c > \lambda'_{c+1}$ for those $c$ for which $\lambda'_{c+1} >1$. Thus we recover Theorem 1.1  of \cite{FHW} (which we have stated as Theorem \ref{FHW} in the Introduction).
	
	(3) Let $k=2$. Then for every $j=1, \dots b$ we have \[\Sigma(2,j)=(b-1)j(a-b+j+1)-\tbinom{j}{2}\tbinom{a-b+j+2}{2}.\] A quick computation yields $\Sigma(2,j)=0$ if and only if $4(b-1)-(j-1)(a-b+j+2)=0$. From Corollary \ref{rowsSp}, $M^{(a,b)'}/SGR^{(a,b)'}{(2,2)} \simeq S^{(a,b)'}$ if and only if $4(b-1)-(j-1)(a-b+j+2)\neq 0$ for all $j=1,\dots,b$. According to Theorem \ref{mainSpe}, if $\lambda$ is a partition with all column lengths greater than or equal to 2, then  \[M^{\lambda} / SGR^\lambda(2,\dots,2) \simeq S^\lambda\]
	if $4(\lambda'_{c+1}-1)-(j-1)(\lambda'_{c}-\lambda'_{c+1}+j+2)\neq 0$ for all $c=1,\dots,\lambda_{1}-1$ and $j=1, \dots, \lambda'_{c+1}$.
\end{exs}				

In the second of the above examples we observed that the expression for the sum $\Sigma(k,j)$ takes on a simple form when $k=b$. We are not aware of a simplification valid for all $k$.

\subsection{Specht modules and Garnir relations}
Recall that the subspace of Garnir relations \\$GR^\lambda(k_1,\dots,k_{\lambda_1-1}) \subseteq M^\lambda$ was defined in the Introduction, after Theorem 1.3. 

Consider the map $\gamma_k \in \Hom _G(\Lambda^a \otimes \Lambda^{k} \otimes \Lambda^{b-k}, \Lambda^a \otimes \Lambda^b)$ of Definition \ref{mapsG}(2). It is straightforward to verify that its image $f(\gamma_k)$ under the Schur functor $f$ has the property that the vector space $\Ima (f(\gamma_k)) \subseteq M^{(a,b)'}$ is spanned by the Garnir relations $g_{1,k}^T$ of (\ref{GarRel}) as $T$ varies over $\mathcal{T}_\lambda$, where $\lambda = (a,b)'$. Thus from  and Corollary \ref{cor2rowsGR} and Theorem \ref{maincorGarnir3}  we obtain the following, which we regard as our third main result of the paper.
\begin{thm}\label{thmmainSpGR} Let  $\lambda$ be a partition of $r$ and $\mu=(\mu_1, \dots, \mu_\ell)=\lambda'$.
	\begin{enumerate}
		\item Suppose $\ell =2$ and  $k$ is a positive integer such that $k \le \mu_2$.  Then as $\mathfrak{S}_{r}$-modules we have $M^{\lambda} / GR^\lambda(k) \simeq S^\lambda$ if and only if $\mu_1>k$ or  $\mu_1=\mu_2=k=1$.
		\item Suppose $k_1,\dots, k_{\ell-1}$ are positive integers such that $k_c \le \mu_{c+1}$, $c=1,\dots,\ell-1$. Then as $\mathfrak{S}_{r}$-modules we have
		$M^{\lambda} / GR^\lambda(k_1,\dots,k_{\ell-1}) \simeq S^\lambda$ if  $\mu_c > \mu_{c+1}$ for those $c$ for which $k_c=\mu_{c+1}>1$.\end{enumerate}
\end{thm}

\begin{rems}
	(1) Suppose $k_1=\dots=k_{\ell-1}=1$. Then the hypothesis ''$\mu_c > \mu_{c+1}$ for those $c$ for which $k_c=\mu_{c+1}>1$'' is trivially satisfied and we conclude from part (2) of the above theorem that $ M^{\lambda} / GR^\lambda(1,\dots,1) \simeq S^\lambda $. Thus we recover the classical presentation of $S^\lambda$ in characteristic zero that we mentioned at the end of the third paragraph of the Introduction (see Exercise 16 of \cite[7.4]{F}).
	
	(2) We note that part (2) of the above theorem shows that the conclusion of \cite[Theorem 1.1]{FHW} holds not only when $k_c=\mu_{c+1}$ for all $c$, but for any choice of $k_c$.
	
	(3) Theorem \ref{thmmainSpGR} provides an answer to the question posed at the end of \cite{FHW}.
\end{rems}

\subsection*{Acknowledgment}
We are indebted to Michelle Wachs for helpful comments concerning \cite{FHW}. Also, we would like to thank the referee for helpful suggestions.

\end{document}